\documentclass[10pt,twoside]{article}
\usepackage[T1]{fontenc} 
\usepackage{amsmath,amsthm}
\usepackage[english]{babel}
\usepackage{graphicx}
\usepackage{amsfonts,amssymb}
\usepackage{hyperref}
\usepackage{stmaryrd}
\usepackage{fancyhdr}
\usepackage{url}
\usepackage{dsfont}

\newtheorem*{conj*}{Conjecture}

\newtheorem*{ack}{Acknowledgements}
\newtheorem*{thm*}{Theorem}

\newtheorem{prop}{Proposition}[section]

\newtheorem{LM}{Lemma}[section]

\newtheorem{thm}{Theorem}[section]
\newtheorem{df}{Definition}[section]
\newtheorem{cor}{Corollary}[section]

\newtheoremstyle{pourlesremarques}{\topsep}{\topsep}{\normalfont}{}{\bfseries}{.}{ }{}
\theoremstyle{pourlesremarques}
\newtheorem{rem}{Remark}[section]
\newtheorem*{rem*}{Remark}
\newtheoremstyle{pourlesexemples}{\topsep}{\topsep}{\normalfont}{}{\bfseries}{.}{ }{}
\theoremstyle{pourlesexemples}

\newcommand{\upequal}{\mathrel{\rotatebox[origin=c]{90}{$=$}}}
\newcommand{\downsimeq}{\mathrel{\rotatebox[origin=c]{-90}{$\simeq$}}}

\renewcommand{\o}{\mathfrak{O}}
\newcommand{\p}{\mathfrak{P}}
\newcommand{\w}{\varpi}
\renewcommand{\d}{\delta}

\renewcommand{\l}{\lambda}
\newcommand{\D}{\Delta}
\newcommand{\C}{\mathbb{C}}

\newcommand{\N}{\mathbb{N}}
\newcommand{\Z}{\mathbb{Z}}
\newcommand{\1}{\mathbf{1}}

\def\YEAR{\year}\newcount\VOL\VOL=\YEAR\advance\VOL by-1995
\def\firstpage{1}\def\lastpage{1000}
\def\received{}\def\revised{}
\def\communicated{}

\makeatletter
\def\magnification{\afterassignment\m@g\count@}
\def\m@g{\mag=\count@\hsize6.5truein\vsize8.9truein\dimen\footins8truein}
\makeatother

\font\eightrm=cmr8
\font\caps=cmcsc10                    
\font\Caps=cmcsc10 scaled \magstep1   

\def\bfseries{\normalsize\caps}

\makeatletter
\@specialpagefalse\headheight=8.5pt
\def\DocMath{}
\renewcommand{\@evenhead}{%
    \ifnum\thepage>\lastpage\rlap{\thepage}\hfill%
    \else\rlap{\thepage}\slshape\leftmark\hfill{\caps\SAuthor}\hfill\fi}%
\renewcommand{\@oddhead}{%
    \ifnum\thepage=\firstpage{\DocMath\hfill\llap{\thepage}}%
    \else{\slshape\rightmark}\hfill{\caps\STitle}\hfill\llap{\thepage}\fi}%
\makeatother

\def\TSkip{\bigskip}
\newbox\TheTitle{\obeylines\gdef\GetTitle #1
\ShortTitle  #2
\SubTitle    #3
\Author      #4
\ShortAuthor #5
\EndTitle
{\setbox\TheTitle=\vbox{\baselineskip=20pt\let\par=\cr\obeylines%
\halign{\centerline{\Caps##}\cr\noalign{\medskip}\cr#1\cr}}%
	\copy\TheTitle\TSkip\TSkip%
\def\next{#2}\ifx\next\empty\gdef\STitle{#1}\else\gdef\STitle{#2}\fi%
\def\next{#3}\ifx\next\empty%
    \else\setbox\TheTitle=\vbox{\baselineskip=20pt\let\par=\cr\obeylines%
    \halign{\centerline{\caps##} #3\cr}}\copy\TheTitle\TSkip\TSkip\fi%
\centerline{\caps #4}\TSkip\TSkip%
\def\next{#5}\ifx\next\empty\gdef\SAuthor{#4}\else\gdef\SAuthor{#5}\fi%
\ifx\received\empty\relax
    \else\centerline{\eightrm Received: \received}\fi%
\ifx\revised\empty\TSkip%
    \else\centerline{\eightrm Revised: \revised}\TSkip\fi%
\ifx\communicated\empty\relax
    \else\centerline{\eightrm Communicated by \communicated}\fi\TSkip\TSkip%
\catcode'015=5}}\def\Title{\obeylines\GetTitle}
\def\Abstract{\begingroup\narrower
    \parskip=\medskipamount\parindent=0pt{\caps Abstract. }}
\def\EndAbstract{\par\endgroup\TSkip}

\long\def\MSC#1\EndMSC{\def\arg{#1}\ifx\arg\empty\relax\else
     {\par\narrower\noindent%
     2000 Mathematics Subject Classification: #1\par}\fi}

\long\def\KEY#1\EndKEY{\def\arg{#1}\ifx\arg\empty\relax\else
	{\par\narrower\noindent Keywords and Phrases: #1\par}\fi\TSkip}

\newbox\TheAdd\def\Addresses{\vfill\copy\TheAdd\vfill
    \ifodd\number\lastpage\vfill\eject\phantom{.}\vfill\eject\fi}
{\obeylines\gdef\GetAddress #1
\Address #2 
\Address #3
\Address #4
\EndAddress
{\def\xs{4.3truecm}\parindent=0pt
\setbox0=\vtop{{\obeylines\hsize=\xs#1\par}}\def\next{#2}
\ifx\next\empty 
     \setbox\TheAdd=\hbox to\hsize{\hfill\copy0\hfill}
\else\setbox1=\vtop{{\obeylines\hsize=\xs#2\par}}\def\next{#3}
\ifx\next\empty 
     \setbox\TheAdd=\hbox to\hsize{\hfill\copy0\hfill\copy1\hfill}
\else\setbox2=\vtop{{\obeylines\hsize=\xs#3\par}}\def\next{#4}
\ifx\next\empty\ 
     \setbox\TheAdd=\vtop{\hbox to\hsize{\hfill\copy0\hfill\copy1\hfill}
                \vskip20pt\hbox to\hsize{\hfill\copy2\hfill}}
\else\setbox3=\vtop{{\obeylines\hsize=\xs#4\par}}
     \setbox\TheAdd=\vtop{\hbox to\hsize{\hfill\copy0\hfill\copy1\hfill}
	        \vskip20pt\hbox to\hsize{\hfill\copy2\hfill\copy3\hfill}}
\fi\fi\fi\catcode'015=5}}\gdef\Address{\obeylines\GetAddress}

\begin{document}
\Title
Essential Whittaker functions for $GL(n)$
\ShortTitle 
\SubTitle   
\Author 
Nadir Matringe
\ShortAuthor 
\EndTitle
\Abstract 
We give a constructive proof of the existence of the essential Whittaker function of a generic representation of $GL(n,F)$, for 
$F$ a non-archimedean local field, using mirabolic restriction techniques. 
\EndAbstract
\MSC 
22E50
\EndMSC
\KEY 
\EndKEY
\Address 
Universit\'e de Poitiers
\Address
Laboratoire de Math\'ematiques et Applications
\Address
T\'el\'eport 2 - BP 30179, Boulevard Marie et Pierre Curie, 86962, Futuroscope Chasseneuil
\Address
matringe@math.univ-poitiers.fr
\EndAddress

\section*{Introduction}

Let $F$ be nonarchimedean local field, we denote by $\o$ its ring of integers, and by 
$\p=\w\o$ the maximal ideal of this ring, where $\w$ is a uniformiser of $F$. We denote by $q$ the cardinality of $\o/\p$ and by $|.|$ the absolute value on $F$ normalised such that $|\w|$ is equal to $q^{-1}$.\\
For $n\geq 1$, we denote the group $GL(n,F)$ by $G_n$, the group $GL(n,\o)$ by $G_n(\o)$, and we set $G_0=\{1\}$. We denote by $A_n$ 
the torus of diagonal matrices in $G_n$, and by $N_n$ the unipotent radical of the Borel subgroup of $G_n$ given by upper triangular matrices. 
For $m\geq 1$, we denote by $K_n(m)$ the subgroup of 
$G_n$, given by matrices $\begin{pmatrix} g & v \\ l & t\end{pmatrix}$, for $g$ in $G_{n-1}(\o)$, $v$ in $\o^{n-1}$, 
$l$ with every coefficient in $\p^m$, and $t$ in $1+\p^m$. We set $K_n(0)=G_n(\o)$.\\

If $\pi$ is a generic representation of $G_2$, the essential vector of $\pi$ was first considered in \cite{C}, for $G_n$ with $n\geq 2$, it 
was studied in \cite{JPS}. Here is one of its main properties: if one calls $d$ the conductor (the power of $q^{-s}$
 in the $\epsilon$ factor with respect to an unramified additive character of $F$) of the representation $\pi$,  
the complex vector space $\pi^{K_n(d)}$ of vectors in $\pi$ fixed under $K_n(d)$, is generated by the essential vector of $\pi$, and 
$\pi^{K_n(k)}$ becomes the null space for $k<d$.\\
However, to prove its existence, one has to study properties of the Rankin-Selberg integrals associated to the pairs 
$(\pi,\pi')$, where $\pi'$ varies through the set of unramified generic representations of $G_{n-1}$.\\ 

We set a few more notations before explaining this.\\
We denote by $\nu$ the positive character $|.|\circ det$ of 
$G_n$.
 We use the product notation for normalised parabolic induction (see Section \ref{parab}). 
 For any sequence of complex numbers $s_1,\dots, s_n$, the representation $|.|^{s_1}\times \dots \times |.|^{s_n}$ of $G_n$ is unramified, 
 and its subspace of $G_n(\o)$-invariant vectors is of dimension $1$ (see Section \ref{genunram}).\\
We choose a character $\theta$ of $(F,+)$ trivial on $\o$ but not on $\p^{-1}$, and use it to define 
a non degenerate character, still denoted $\theta$, of the standard unipotent subgroup $N_n$ of $G_n$, by 
$\theta(n)= \theta(\sum_{i=1}^{n-1} n_{i,i+1}$).\\  

For $n\geq 2$, let $\pi$ and $\pi'$ be representations of Whittaker type (see Section \ref{genandL}) of $G_n$ and $G_{n-1}$ respectively, and denote by 
$W(\pi,\theta)$ and $W(\pi',\theta^{-1})$ their respective Whittaker models (which are quotients of $\pi$ and $\pi'$) with respect to $\theta$ and $\theta^{-1}$.\\
If $W$ and $W'$ belong respectively to $W(\pi,\theta)$ and $W(\pi',\theta^{-1})$, we denote $I(W,W',s)$ the associated 
Rankin-Selberg integral (see Section \ref{genandL}).\\
For example, for a sequence of complex numbers $a_1,\dots, a_m$, 
the induced representation $|.|^{a_1}\times \dots \times |.|^{a_m}$ of $G_m$, is of Whittaker type. 
If moreover $Re(a_1)\geq \dots \geq Re(a_m)$, the representation $|.|^{a_1}\times \dots \times |.|^{a_m}$ is of Langlands' type, 
and its Whittaker model contains a unique normalised spherical Whittaker function $W(q^{-a_1},\dots,q^{-a_m})$. It is the unique Whittaker 
function on $G_m$, fixed by $G_m(\o)$, which equals $1$ on $G_m(\o)$, and 
associated to the Satake parameter $\{q^{-a_1},\dots,q^{-a_m}\}$ (see \cite{S}). For fixed $g$ in $G_m$, the function  
$W(q^{-s_1},\dots,q^{-s_m})(g)$ is an element of the ring $\C[q^{\pm s_1},\dots, q^{\pm s_m}]^{\mathcal{S}_m}$ of invariant Laurent polynomials.
To define the essential vector of $\pi$, one needs to show as in \cite{JPS}, the following theorem 
(see \cite{GJ} for the definition of the $L$ function of an irreducible representation of $G_n$):

\begin{thm*}
Let $\pi$ be a generic representation of $G_n$ with Whittaker model $W(\pi,\theta)$, then there exists 
in $W(\pi,\theta)$ a unique $G_{n-1}(\o)$-invariant function $W_{\pi}^{ess}$, such that for every sequence of 
complex numbers $s_1,\dots,s_{n-1}$, one has the equality
$I(W_{\pi}^{ess},W(q^{-s_1},\dots,q^{-s_{n-1}}),s)= \prod_{i=1}^{n-1}L(\pi,s+s_i)$.
\end{thm*}

 Hence, the statement of the 
theorem is equivalent to say that for 
any unramified representation $\pi'$ of Langlands' type of $G_{n-1}$ with 
normalised spherical Whittaker function $W_{\pi'}^0$ in $W(\pi',\theta)$, one has the equality $I(W_{\pi}^{ess},W_{\pi'}^0,s)= L(\pi,\pi',s)$ (see Section \ref{genandL} for the definition of 
$L(\pi,\pi',s)$ and the equality $L(\pi,\pi',s)=\prod_{i=1}^{n-1}L(\pi,s+s_i)$ when 
$\pi'=|.|^{s_1}\times \dots \times |.|^{s_m}$).\\
Using this theorem, it is then shown in \cite{JPS}, using the functional equation of $L(\pi,\pi',s)$,
 that the space $W(\pi,\theta)^{K_n(d)}$ is a complex line spanned by $W_{\pi}^{ess}$, and that $W(\pi,\theta)^{K_n(k)}$ is zero for $k<d$.\\ 

In this paper, we will show the following result, using the interpretation in terms of restriction of Whittaker 
functions of the Bernstein-Zelevinsky derivatives.\\ 
Let $\pi$ be a ramified generic representation of $G_n$, and $\pi_u$ be the unramified component of the first nonzero spherical Bernstein-Zelevinsky 
derivative $\pi^{(n-r)}$ of $\pi$ (see Definition \ref{piu} for the precise definition). The representation $\pi_u$ is an unramified representation 
of Langlands' type of $G_r$ when $r\geq 1$. In this situation, we show in Corollary \ref{corformule}, that there is a unique Whittaker function $W_{\pi}^{ess}$ in $W(\pi,\theta)$, which is right 
$G_{n-1}(\o)$-invariant, and which satisfies, for $a=diag(a_1,\dots,a_{n-1})\in A_{n-1}$ and $a'=diag(a_1,\dots,a_{r})\in A_r$, the equality:
\begin{equation} \label{formule}W(diag(a,1))=W_{\pi_u}^0(a')\nu(a')^{(n-r)/2}\1_{\o}(a_r)\prod_{r<i<n}\1_{\o^*}(a_i),\end{equation}
 when $r\geq 1$, and
  \begin{equation} \label{formule'}W(diag(a,1))=\prod_{0<i<n}\1_{\o^*}(a_i)\end{equation} when $r=0$.\\
Computing the integral $I(W_{\pi}^{ess},W_{\pi'}^0,s)$ for an unramified representation $\pi'$ of Langlands' type of $G_{n-1}$, we 
will obtain in Corollary \ref{testf} the statement (more precisely a slightly more general statement) of the theorem stated above.\\

For $GL(2,F)$, a detailed account about newforms can be found in \cite{Sc}, the author obtains Formula (\ref{formule}) (see Section 2.4 of [loc. cit.]) 
up to normalisation by an $\epsilon$-factor. For $GL(n,F)$, Miyauchi (\cite{Mi}) recently obtained Formula (\ref{formule}), assuming the existence 
of the essential vector, by using Hecke algebras, i.e. generalising 
Shintani's method for spherical representations. 

\begin{rem*} The reason why we got interested in reproving the existence of such a vector is the following. 
In \cite{JPS}, the uniqueness of such a vector is proved. 
The proof of the existence is valid only for generic representations $\pi$ appearing as subquotients of representations parabolically 
induced by ramified characters of $GL(1,F)$ and cuspidal representations of $GL(r,F)$ for $r\geq 2$, i.e. generic representations whith 
$L$-function equal to one.\\
Before we explain this, let us mention that Jacquet (see \cite{J}) found a simple fix for the proof of \cite{JPS}, so that 
the motivation of writing our note is really to give a constructive proof of the existence of this vector, which provides a nice 
application of the techniques developed in \cite{CP}.\\

In \cite{JPS}, the following is shown: for fixed $W$ in $W(\pi,\theta)$, the function 
$$P(W,q^{-s_1},\dots,q^{-s_{n-1}})=I(W,W(q^{-s_1},\dots,q^{-s_{n-1}}),0)/\prod_{i=1}^{n-1}L(\pi,s_i)$$ belongs to the ring 
$\C[q^{\pm s_1},\dots,q^{\pm s_{n-1}}]^{\mathcal{S}_{n-1}}$ of symmetric Laurent polynomials in the variables $q^{-s_i}$. 
It is also shown that the existence of the essential vector is equivalent to the fact that the vector space 
$$I(\pi)=\{P(W,q^{-s_1},\dots,q^{-s_{n-1}}), W\in W(\pi,\theta)\},$$ which is actually an ideal, is equal to the the full ring  
$$\C[q^{\pm s_1},\dots,q^{\pm s_{n-1}}]^{\mathcal{S}_{n-1}}.$$\\
The argument used to prove it goes like this:\\
For $W$ well chosen, $P(W,q^{\-s_1},\dots,q^{-s_{n-1}})$ is equal to 
$$\prod_{i=1}^{n-1} 1/L(\pi,s_i).$$   
We denote by $Q$ the element $1/L(\pi,s)$ of $\C[q^{-s}]$, so that 
$$P(W,q^{-s_1},\dots,q^{-s_{n-1}})= \prod_{i=1}^{n-1}Q(q^{-s_i}).$$
Because of the functional equation of the $L$-function $L(\pi,|.|^{s_1}\times \dots \times |.|^{s_{n-1}},s)$, denoting 
$\pi^{\vee}$ the smooth contragredient of $\pi$, one shows 
that $I(\pi)$ also contains the product $\prod_{i=1}^{n-1}Q'(q^{-1}q^{s_i})$, where 
$$Q'(q^{-s})= 1/L(\pi^{\vee},s).$$ 
Proposition 2.1. of the paper then shows that $Q'(q^{-1}q^{s})$ and 
$Q(q^{-s})$ are prime to one another in $\C[q^{\pm s}]$, and they deduce from this that no maximal ideal
$$I_{q^{-a_1},\dots,q^{-a_{n-1}}}= \{R \in \C[q^{\pm a_1},\dots,q^{\pm a_{n-1}}], R(q^{\pm a_1},\dots,q^{\pm a_{n-1}})=0\}$$ for 
 $(a_1,\dots,a_{n-1})$ in ${\C}^{n-1}$, contains $\prod_{i=1}^{n-1}Q'(q^{-1}q^{s_i})$ and $\prod_{i=1}^{n-1}Q(q^{-s_i})$ 
together, which implies the result.\\
This last step is false as soon as $n\geq 3$, and there are $a$ and $b$ in $\C^*$ such that 
$Q(a)=Q'(q^{-1}b^{-1})=0$, because then both products belong to any ideal $I_{a,b,\dots,x_{n-1}}$. 
This is the case as soon as the degree $d^\circ(Q)$ of $Q$ satisfies $d^\circ(Q)\geq 1$.\\ 
However, using the functional equation of $L(\pi,|.|^{z_1}\times \dots \times |.|^{z_{n-1}})$, and the cyclicity of 
$W(q^{-z_1},\dots,q^{-z_{n-1}})$ in $W(|.|^{z_1}\times \dots \times |.|^{z_{n-1}},\theta)$ when $Re(z_i)\geq Re(z_{i+1})$, 
Jacquet noticed (see \cite{J}) that one can find for every $(a_1,\dots,a_{n-1})$ in ${\C}^{n-1}$, a polynomial in $I(\pi)$, taking the value 
$1$ when evaluated at $(q^{\pm a_1},\dots,q^{\pm a_{n-1}})$, so $I(\pi)$ is indeed equal to 
$\C[q^{\pm s_1},\dots,q^{\pm s_{n-1}}]^{\mathcal{S}_{n-1}}$.
\end{rem*}

\section{Preliminaries}\label{rappelgen}

In this section, we first recall basic facts about smooth representations of locally profinite groups. We 
then focus on $G_n$, recall results from \cite{BZ} about derivatives, then introduce the 
$L$-function of a pair of representations of Whittaker type, we discuss espacially the unramified case.

\subsection{Smooth representations, restriction and induction}

When $G$ is an $l$-group (locally compact totally disconnected group), we denote by 
$Alg(G)$ the category of smooth complex $G$-modules. We denote by $\widehat{G}$ the group of smooth characters (smooth representations of dimension $1$) of $G$.
If $(\pi,V)$ belongs to $Alg(G)$, $H$ is a closed subgroup of $G$,
 and $\chi$ is a character of $H$, we denote by $V(H,\chi)$ the subspace of $V$ generated by vectors of the form $\pi(h)v-\chi(h)v$ 
for $h$ in $H$ and $v$ in $V$. 
This space is stable under the action of the subgroup $N_G(\chi)$ of the normalizer $N_G(H)$ of $H$ in $G$, which fixes $\chi$.\\
We denote by $\delta_G$ the positive character of $G$ such that if $\mu$ is a right Haar measure on $G$, and $int$ is the action 
of $G$ on smooth functions $f$ with compact support in $G$, given by $(int(g)f)(x)=f(g^{-1}xg)$, then 
$\mu \circ int(g)= \delta_G(g)\mu $ for $g$ in $G$.\\ 
The space $V(H,\chi)$ is $N_G(\chi)$-stable. Thus, if $L$ is a closed-subgroup of $N_G(\chi)$, and $\delta'$ is a (smooth) character of 
$L$ (which will be a normalising character dual to that of normalised induction later), the quotient 
$V_{H,\chi}=V/V(H,\chi)$ (that we simply denote by $V_H$ when 
$\chi$ is trivial) becomes a smooth $L$-module for the (normalised) action $l.(v + V(H,\chi))= \delta'(l)\pi(l)v + V(H,\chi)$ of $L$ on 
$V_{H,\chi}$.\\  
We denote by $V^H$ the subspace of vectors of $V$ fixed by $H$; for $H$ compact and open, the functor 
$V\mapsto V^H$ from $Alg(G)$ to $Alg(G_0)$ is exact (\cite{BH}, 2.3., Corollary 1).\\
 We say that $(\pi,V)$ in $Alg(G)$ is \textit{admissible} if for any 
compact open subgroup $H$ of $G$, the vector space $V^H$ is finite dimensional.\\ 
If $H$ is a closed subgroup of an $l$-group $G$, and $(\rho,W)$ belongs to $Alg(H)$, we define the objects 
$(ind_H^G(\rho), V_c=ind_H^G(W))$ and $(Ind_H^G(\rho), V=Ind_H^G(W))$ of $Alg(G)$ as follows. The space $V$ is the space of smooth functions 
from $G$ to $W$, fixed under right translation by the elements of a compact open subgroup 
$U_f$ of $G$, and  satisfying $f(hg)=\rho(h)f(g)$ for all $h$ in $H$ and $g$ in $G$. The space $V_c$ 
is the subspace of $V$, consisting of functions with support compact mod $H$, in both cases, the action of $G$ is by right 
translation on the functions.\\
We recall that by Frobenius reciprocity law (\cite{BH}, 2.4.), the spaces $Hom_G(\pi,Ind_H^G(\rho))$ and $Hom_H(\pi_{|H},\rho)$ are isomorphic when 
$\pi$ (resp. $\rho$) belongs to $Alg(G)$ (resp. $Alg(H)$).\\
If the group $G$ is exhausted by compact subsets (which is the case of closed subgroups of $G_n$), 
and $(\pi,V)$ is irreducible, it is known (\cite{BH}, 2.6., Corollary 1)  that the center $Z$ of $G$ acts on $V$ by the so-called central
 character of $\pi$ which we will denote $c_\pi$. When $G=G_n$, then $Z$ identifies with $F^*$. By definition, the real part 
$Re(\chi)$ of a character $\chi$ of $F^*$ is the real number $r$ such that $|\chi(t)|_{\C}=|t|^r$, where $|z|_{\C}=\sqrt{z\bar{z}}$ for $z$ 
in $\C$. 

\subsection{Parabolic induction and segments for GL(n)}\label{parab}

Now we focus on the case $G=G_n$, we will only consider smooth representations of its closed subgroups. 
It is known that irreducible representations of $G_n$ are admissible (see \cite{C2}).\\
If $n\geq 1$, let $\bar{n}=(n_1,\dots,n_t)$ be a partition of $n$ of length $t$ (i.e. an ordrered
 set of $t$ positive integers 
whose sum is $n$), we denote by $M_{\bar{n}}$ to be the Levi subgroup of $G_n$, of matrices 
$diag(g_1,\dots,g_t)$, with each $g_i$ in $G_{n_i}$, by $N_{\bar{n}}$ the unipotent subgroup 
of matrices $\begin{pmatrix} I_{n_1} & \star & \star \\ & \ddots & \star \\ & & I_{n_t} \end{pmatrix}$, and by $P_{\bar{n}}$ the standard 
parabolic subgroup 
$M_{\bar{n}}N_{\bar{n}}$ (where $M_{\bar{n}}$ normalises $N_{\bar{n}}$). Note that $M_{(1,\dots,1)}$ is equal to $A_n$, and 
$N_{(1,\dots,1)}=N_n$. For each $i$, let $\pi_i$ be a smooth representation of $G_{n_i}$, then the tensor product 
$\pi_1 \otimes \dots \otimes \pi_t$ is a representation of $M_{\bar{n}}$, which can be considered as a representation of $P_{\bar{n}}$ trivial on 
$N_{\bar{n}}$. We will use the product notation  
$$\pi_1\times \dots \times \pi_t= Ind_{P_{\bar{n}}}^{G_n}(\d_{P_{\bar{n}}}^{1/2}\pi_1 \otimes \dots \otimes \pi_t )$$ for 
the normalised parabolic induction. Parabolic induction preserves finite length and admissibility (see \cite{BZ} or \cite{C2}).\\
We say that an irreducible 
representation $(\rho,V)$ of $G_n$ is cuspidal, if the Jacquet module $V_{N_{\bar{n}}}$ is zero whenever
 $\bar{n}$ is a proper partition of $n$ (i.e. we exclude $\bar{n}=(n)$).\\
 Suppose that $\bar{n}=(m,\dots,m)$ is a partition of $n$ of length $l$, and that $\rho$ is a cuspidal representation of $G_m$. Then 
Theorem  9.3. of \cite{Z} implies that the $G_n$-module $\nu^{-(l-1)}\rho\times \nu^{-(l-2)}\rho\times \dots \times\nu^{-1}\rho\times \rho$ has a unique irreducible quotient 
which we denote $[\nu^{-(l-1)}\rho,\nu^{-(l-2)}\rho, \dots, \nu^{-1}\rho,\rho].$ We will call such a representation a segment, 
it is known that segments are the quasi square integrable representations 
of $G_n$, but we won't need this result.\\
We end this paragraph with a word about induced representations of Langlands' type:

\begin{df}
Let $\D_1,\dots,\D_t$ be segments of respectively $G_{n_1},\dots,G_{n_t}$, and suppose that 
$Re(c_{\D_i})\geq Re(c_{\D_{i+1}})$. Let $n=n_1+\dots+n_t$, then 
the representation $\D_1\times \dots \times \D_t$ of $G_n$ is said to be induced of Langlands' type.  
\end{df}

These representations enjoy many remarkable properties, some of which we will recall later, here is a first one (which is the main result of 
\cite{Sil}).

\begin{prop}\label{Lgtype}
 Let $\pi$ be induced of Langlands' type, then $\pi$ has a unique irreducible quotient $Q(\pi)$. Moreover, considering that isomorphic representations are equal, the map 
 $\pi\mapsto Q(\pi)$ gives a bijection between the set of induced representations of Langlands' type of $G_n$, and the set of 
 irreducible representations of $G_n$.
\end{prop}

\subsection{Berstein-Zelevinsky derivatives}\label{BZder}

For $n\geq2$ we denote by $U_{n}$ the group of matrices of the form $\begin{pmatrix} 
                                                                                         I_{n-1}    & v \\
                                                                                                    & 1 \end{pmatrix}$.\\

For $n>k\geq 1$, the group $G_{k}$ embeds naturally in $G_{n}$, and is given by matrices of the form 
$diag(g,I_{n-k})$. We denote by $P_n$ the mirabolic subgroup $G_{n-1}U_n$ of $G_n$ for $n\geq 2$, and $P_1=\{1_{G_1}\}$. If one sees $P_{n-1}$ as a subgroup of $G_{n-1}$ 
itself embedded in $G_n$, then $P_{n-1}$ is the normaliser of $\theta_{|U_n}$ in $G_{n-1}$ (i.e. if $g\in G_{n-1}$, then 
$\theta (g^{-1}ug)=\theta(u)$ for all $u\in U_n$ if and only if $g\in P_{n-1}$). We define the following functors:\\

\begin{itemize}

 \item The functor $\Phi^{-}$ from $Alg(P_k)$ to $Alg(P_{k-1})$ such that, if $(\pi,V)$ is a smooth $P_k$-module, 
$\Phi^{-} V =V_{U_k,\theta}$, and $P_{k-1}$ acts on $\Phi^{-}(V)$ by
 $\Phi^{-} \pi (p)(v+V(U_k,\theta))= \delta_{P_k} (p)^{-1/2}\pi (p)(v+V(U_k,\theta))$.

\item The functor $\Phi^{+}$ from $Alg(P_{k-1})$ to $Alg(P_{k})$ such that, for $\pi$ in $Alg(P_{k-1})$, one has
$\Phi^{+} \pi = ind_{P_{k-1}U_k}^{P_k}(\delta_{P_k}^{1/2}\pi \otimes \theta)$.

\item The functor $\hat{\Phi}^{+}$ from $Alg(P_{k-1})$ to $Alg(P_{k})$ such that, for $\pi$ in $Alg(P_{k-1})$, one has
$\hat{\Phi}^{+} \pi = Ind_{P_{k-1}U_k}^{P_k}(\delta_{P_k}^{1/2}\pi \otimes \theta)$.

\item The functor $\Psi^{-}$ from $Alg(P_k)$ to $Alg(G_{k-1})$, such that if $(\pi,V)$ is a smooth $P_k$-module, 
$\Psi^{-} V =V_{U_k,1}$, and $G_{k-1}$ acts on $\Psi^{-}(V)$ by
 $\Psi^{-} \pi (g)(v+V(U_k,1))= \delta_{P_k} (g)^{-1/2}\pi (g)(v+V(U_k,1))$.

\item The functor $\Psi^{+}$ from $Alg(G_{k-1})$ to $Alg(P_{k})$, such that for $\pi$ in $Alg(G_{k-1})$, one has
$\Psi^{+} \pi = ind_{G_{k-1}U_k}^{P_k}(\delta_{P_k}^{1/2}\pi \otimes 1)=\delta_{P_k}^{1/2}\pi \otimes 1 $.

\end{itemize}
 
These functors have the following properties which can be found in \cite{BZ}:

\begin{prop}\label{baseder}
a) The functors $\Phi^{-}$, $\Phi^{+}$, $\Psi^{-}$, and $\Psi^{+}$ are exact.\\
b) $\Psi^{-}$ is left adjoint to $\Psi^{+}$.\\
b') $\Phi^{-}$ is left adjoint to $\hat{\Phi}^{+}$.\\
c) $\Phi^{-}\Psi^{+}=0$ and $\Psi^{-}\Phi^{+}=0$.\\
d) $\Psi^{-}\Psi^{+}\simeq Id$ and $\Phi^{-}\Phi^{+}\simeq Id$.\\
e) One has the exact sequence $0\rightarrow \Phi^{+}\Phi^{-} \rightarrow Id \rightarrow \Psi^{+}\Psi^{-} \rightarrow 0$.\\
\end{prop}

Following \cite{CP}, if $\tau$ belongs to $Alg(P_n)$, we will denote $(\Phi^{-})^{k} \tau$ by $\tau_{(k)}$, 
and as usual, $\tau^{(k)}$ will be defined as $\Psi^{-}\tau_{(k-1)}$.\\

Because of e), $\tau$ has a natural filtration of $P_n$-modules 
$0\subset \tau_{n} \subset \dots \subset \tau_{1}=\tau$, where $\tau_k= {\Phi^{+}}^{k-1} {\Phi^{-}}^{k-1}\tau$. 
We will use the notation $\tau_{(k),i}$ for $(\tau_{(k)})_i$.  
The following observation is just a restatement of the definitions:

\begin{LM}\label{filtr}
If $\tau$ belongs to $Alg(P_n)$, then $\tau_k= \Phi^{+}(\tau_{(1),k-1})$ for $k\geq 1$.
\end{LM} 

\subsection{Representations of Whittaker type and their L-functions}\label{genandL}

We recall that we fixed a character $\theta$ of conductor $\o$ in the introduction. 

\begin{df}
Let $\pi$ be an admissible representation of $G_n$, 
we say that $\pi$ is of Whittaker type $\pi$ if $Hom(\pi,Ind_{N_n}^{G_n}(\theta))$ is of dimension $1$, or equivalently, 
according to Frobenius reciprocity law, 
if the space $Hom_{N_n}(\pi,\theta)$ is of dimension $1$. We denote by $W(\pi,\theta)$ the image of $\pi$ in $Ind_{N_n}^{G_n}(\theta)$, it is 
called the Whittaker model of $\pi$ (with respect to $\theta$), it is a quotient of $\pi$.
\end{df}

Being of Whittaker type does not depend on the character $\theta$ of $(F,+)$, as another non trivial character $\theta'$ of $(F,+)$ 
will give birth to a character $\theta'$ of $N_n$, conjugate to $\theta$ by $A_n$.\\

In terms of derivatives, as the representation $Ind_{N_n}^{P_n}(\theta)$ is isomorphic to 
$(\hat{\Phi}^{+})^{n-1}\Psi^+(\1)$, where $\1$ is the trivial representation of $G_0$, applying 
$b)$ and $b')$ of Proposition \ref{baseder}, we obtain that $Hom_{N_n}(\pi,\theta)\simeq \C$ if and only if $\pi^{(n)}=\1$. Applying this 
to product of segments, and using the rules of ``derivation`` given in Lemma 3.5 of \cite{BZ} and Proposition 9.6. of \cite{Z}, we obtain 
that if $\D_1,\dots, \D_t$ are segments of $G_{n_1},\dots,G_{n_t}$ respectively, the representation 
$\pi=\D_1\times \dots \times \D_t$ of $G_n$ (for $n=n_1+\dots+n_t$) is of Whittaker type. If the segments $\D_i$ are ordered so that 
$\pi$ is of Langlands' type, we can say more according to the main result of \cite{JS3}.

\begin{prop}\label{langlandswhittaker}
For $n\geq 1$, let $\pi$ be a representation of $G_n$, which is induced of Langlands' type, then it has an injective 
Whittaker model, i.e. $\pi\simeq W(\pi,\theta)$ (equivalently $\pi$ embeds in $Ind_{N_n}^{G_n}(\theta)$).
\end{prop}

If $\pi$ is irreducible and embeds in $Ind_{N_n}^{G_n}(\theta)$, 
it is a well-known theorem of Gelfand and Kazhdan (\cite{GK}) that the multiplicity of $\pi$ in $Ind_{N_n}^{G_n}(\theta)$ is $1$, 
we then say that $\pi$ is \textit{generic}. We recall (Theorem 9.7 of \cite{Z}), that every generic representation $\pi$ of $GL(n,F)$ 
can be written uniquely, 
up to permutation of the terms in the product, as a commutative product of unlinked (see 4.1. of \cite{Z}) segments 
$$[\nu^{-(k_1(\pi)-1)}\rho_1(\pi), \dots,\rho_1(\pi)]\times \dots \times [\nu^{-(k_t(\pi)-1)}\rho_t(\pi), \dots,\rho_t(\pi)].$$ 
In particular, generic representations are the representations of Langlands' type which are irreducible.\\

We now recall from \cite{JPS2}, some results about the $L$-function of a pair of representations of Whittaker type. Let 
$\pi$ be a representation of $G_n$ of Whittaker type, and $\pi'$ be a representation of Whittaker type of $G_m$, 
with respective Whittaker models $W(\pi,\theta)$ and $W(\pi',\theta^{-1})$, for $n\geq m\geq 1$.\\ 
When $n>m$, and $W$ and $W'$ are respectively in $W(\pi,\theta)$ and $W(\pi',\theta^{-1})$, 
we write $$I(W,W',s)=\int_{N_m\backslash G_m} W \begin{pmatrix} g & \\ & I_{n-m} \end{pmatrix} W'(g)\nu(g)^{s-(n-m)/2} dg.$$ 
When $n=m$, and $W$ and $W'$ are respectively in $W(\pi,\theta)$ and $W(\pi',\theta^{-1})$, $\phi$ is in 
$\mathcal{C}_c^{\infty}(F^n)$, and $\eta$ is the row vector $(0,\dots,0,1)$ in the space 
$\mathcal{M}(1,n,F)$ of row matrices $1$ by $n$ with entries in $F$, we write
$$I(W,W',\phi, s)=\int_{N_n\backslash G_n} W (g) W'(g) \phi(\eta g) \nu(g)^{s} dg.$$
It is shown in \cite{JPS2} that these integrals converge absolutely for $Re(s)$ large, and define elements of $\C(q^{-s})$.\\  
If $n>m$, the integrals $I(W,W', s)$ (which we shall also write $I(W',W, s)$ when convenient) span, when $(W,W')$ varies in 
$W(\pi,\theta)\times W(\pi ',\theta^{-1})$, a fractional ideal of $\C[q^{s},q^{-s}]$, which is 
generated by a unique Euler factor $L(\pi,\pi',s)$.
If $n=m$, the integrals $I(W,W',\phi, s)$ span, when $(W,W',\phi)$ varies in 
$W(\pi,\theta)\times W(\pi ',\theta^{-1}) \times \mathcal{C}_c^\infty(F^n)$, a fractional ideal of $\C[q^{s},q^{-s}]$, which is 
generated by a unique Euler factor $L(\pi,\pi',s)$. If $n<m$, we define $L(\pi,\pi',s)$ to be $L(\pi',\pi,s)$.\\
We recall Proposition 9.4 of \cite{JPS2}.

\begin{prop}\label{pairlanglands}
For $n\geq m \geq 1$, if $\pi=\D_1\times \dots \times \D_t$ is a representation of $G_n$ induced of Langlands' type, and 
$\pi'=\D_1'\times \dots \times \D_u'$ is a representation of $G_m$ induced of Langlands' type, then 
$L(\pi,\pi',s)=\prod_{i,j}L(\D_i,\D_j',s)$.
\end{prop}

Finally, we recall that it is proved in Section 5 of \cite{JPS2}, that if $\pi$ is a generic representation of $G_n$, 
and $\chi$ is a character of $G_1$, one has 
the equality $L(\pi,\chi,s)=L(\chi\otimes \pi,s)$ between the Rankin-Selberg $L$-function on the left, and the Godement-Jacquet 
$L$-function on the right. 

\subsection{Unramified representations}\label{genunram}

We say that a representation $(\pi,V)$ of $G_n$ is \textit{unramified} (or \textit{spherical}), if it admits a nonzero $G_n(\o)$-fixed vector in its space.
If it is the case, we recall that the Hecke convolution algebra $\mathcal{H}_n$
 (whose elements are the functions with compact support on $G_n$, which are left and right invariant under $G_n(\o)$), acts on $V^{G_n(\o)}$ (when $V^{G_n(\o)}$ is of dimension $1$, the action is necessarily by a character). 
The Hecke algebra $\mathcal{H}_n$ is commutative and isomorphic by the Satake isomorphism to the algebra 
$\C[X_1^{\pm 1},\dots,X_n^{\pm 1}]^{\mathcal{S}_{n}}$ according to \cite{Sat}. Hence, a character of $\mathcal{H}_n$ is associated to a unique set of 
nonzero complex numbers $\{z_1,\dots,z_n\}$, corresponding to the evaluation $P\mapsto P(z_1^{\pm 1},\dots,z_n^{\pm 1})$ from 
$\C[X_1^{\pm 1},\dots,X_{n}^{\pm 1}]^{\mathcal{S}_{n}}$ to $\C$. It is known that when $\pi$ is unramified and irreducible (see Section 4.6 in \cite{Bu} for example), then $V^{G_n(\o)}$ is one dimensional, and that the corresponding character of $\mathcal{H}_n$ determines $\pi$, in which case 
the associated set of nonzero complex numbers is called the \textit{Satake parameter} of $\pi$.

We recall with proofs, some classical facts about parabolically induced spherical representations.

\begin{prop}\label{parabspherical}
 Let $\pi_i$ be a representation of $G_{n_i}$ for $i$ between $1$ and $t$, $n=n_1+\dots+n_t$, $M_{\bar{n}}$ be the standard Levi subgroup of 
$n$ corresponding to the partition $(n_1,\dots,n_t)$, and $\pi=\pi_1\times \dots \times \pi_t$. The vector space 
$Hom_{G_n(\o)}(\1,\pi)$ is isomorphic to $Hom_{M_{\bar{n}}\cap G_n(\o)}(\1,\pi_1\otimes \dots \otimes \pi_t)$. 
\end{prop}
\begin{proof}
From Mackey's theory (\cite{BZ}, Theorem 5.2.), as $G_n$ is equal to the double-class $P_{\bar{n}}G_n(\o)$, the restriction of $\pi$ to 
$G_n(\o)$ is equal to $Ind_{M_{\bar{n}}\cap G_n(\o)}^{G_n(\o)}(\pi_1\otimes \dots \otimes \pi_t)$. The result then 
follows from Frobenius reciprocity law.
\end{proof}

We have the following corollaries to this.

\begin{cor}\label{sphericalsubquotient}
Let $\chi_1,\dots,\chi_n$ be unramified characters of $F^*$, then the representation $\pi=\chi_1\times \dots \times \chi_n$ of $G_n$ 
is unramified, and $\pi^{G_n(\o)}$ is one dimensional. In particular, by the exactness of the functor 
$V\mapsto V^{G_n(\o)}$ from $Alg(G_n)$ to $Alg(G_0)$, the representation $\pi$ has only one irreducible spherical subquotient.
\end{cor}

When $\pi=|.|^{s_1}\times \dots \times |.|^{s_n}$ is an unramified representation of $G_n$ induced from the Borel subgroup, we just saw that $\pi^{G_n(\o)}$ 
is of dimension $1$. The character of $\mathcal{H}_n$, given by its action on $\pi^{G_n(\o)}$ corresponds to the set $\{q^{-s_1},\dots,q^{-s_n}\}$. 
If $\pi$ 
is moreover of Langlands' type, it is the Satake parameter of $Q(\pi)$ tahnks to the next corollary.

\begin{cor}\label{langlandspherical}
If $\pi=\chi_1\times \dots \times \chi_n$ is as above, but moreover of Langlands' type, then it is the irreducible quotient of $\pi$ which 
is spherical. In particular, $\pi$ is spanned by $\pi^{G_n(\o)}$.
\end{cor}
\begin{proof}
One checks by induction on $n$ the following assertion: the representation $\pi$ can be written 
$\pi_1\times \dots \times \pi_t$, where for each $i$, the representation $\pi_i$ is equal to 
$\nu^{(l_i -1)/2} \mu_i\times \nu^{(l_i -3)/2} \mu_i\dots \times  \nu^{(1-l_i)/2}\mu_i$ for a positive 
integer $l_i$ and an unramified character $\mu_i$ of $F^*$, and such that the segments 
$[\pi'_i]=[\nu^{(1-l_i)/2}\mu_i,\dots,\nu^{(l_i -1)/2} \mu_i]$ are unlinked, the character $\nu^{(1-l_t)/2}\mu_t$ is equal 
to $\chi_n$, and the segment 
$[\pi'_t]$ contains any other $[\pi'_i]$ in which $\chi_n$ occurs (i.e. if $\chi_n=\nu^r \mu_i$ for $r\in \frac{1}{2}\Z$ between 
$(1-l_i)/2$ and $(l_i-1)/2$, then $r=(1-l_i)/2$ and $l_i \leq l_t$). From \cite{Z}, Section 2, the 
irreducible quotient of $\pi_i$ is the irreducible submodule of 
$\nu^{(1-l_i)/2} \mu_i\times \dots \times \nu^{(l_i -1)/2} \mu_i$, i.e. the character 
$\tilde{\mu_i}=\mu_i\circ det$ of $G_{l_i}$. The representation 
$\tilde{\mu_1}\times \dots \times\tilde{\mu_l}$ is thus a quotient of $\pi$, which is irreducible by Theorem 
4.2 of \cite{Z}, hence it is its irreducible quotient. It is spherical from Proposition \ref{parabspherical}, as the representations $\tilde{\mu_i}$ are spherical. 
\end{proof}

\begin{cor}\label{segmentramified}
A segment $\D$ of $G_n$, for $n\geq 2$, is always ramified. 
\end{cor}
\begin{proof}
If $\D$ was unramified, as it is irreducible, its Satake parameter would be equal to a set $\{q^{-s_1},\dots,q^{-s_n}\}$, hence the same 
as that of $Q(\pi)$, for $\pi=|.|^{s_1}\times \dots \times |.|^{s_n}$ (if we order the $s_i$'s in a correct way). In particular, 
$\D$ would be equal to $Q(\pi)$, which is absurd according to Proposition \ref{Lgtype}.
\end{proof}
 
For unramified representations of Langlands' type, normalised spherical Whittaker functions are test functions for $L$ factors. 
If $\pi$ is an unramified representation of Langlands' type, we denote by $W_{\pi}^0$ the spherical function in $W(\pi,\theta)$ which is equal to $1$ 
on $G_n(\o)$, and call it the normalised spherical Whittaker function of $\pi$. 
In \cite{S}, Shintani gave an explicit formula for $W_{\pi}^0$ in terms of the Satake parameter of $\pi$.\\
 Using this formula, Jacquet and Shalika found (Proposition 2.3 of \cite{JS}, and equality (3) in Section 1 of 
\cite{JS2}), for $\pi$ and $\pi'$ two unramified representations of Langlands' type (see the discussion after Equations 
 (\ref{test1}) and (\ref{test2}) below) of 
$G_n$ and $G_m$ respectively, and for correct normalisations of Haar measures, the equalities: 
$$L( \pi,\pi',s)=I( W_{\pi}^0,W_{\pi'}^0,\mathbf{1}_{\o^{n}},s)$$ 
\begin{equation}\label{test1}=\int_{A_{n}}W_{\pi}^0(a)W_{\pi'}^0(a)\mathbf{1}_{\o}(a_{n})\delta_{B_{n}}(a)^{-1}\nu(a)^{s}d^*a,\end{equation} when 
$n=m$, and $$L( \pi,\pi',s)=I(W_{\pi}^0,W_{\pi'}^0 ,s)$$ 
\begin{equation}\label{test2}=\int_{A_m}W_{\pi}^0 \begin{pmatrix}a & \\ & I_{n-m}\end{pmatrix}
W_{\pi'}^0(a)\delta_{B_m}(a)^{-1}\nu(a)^{s-(n-m)/2}d^*a\end{equation}
 when $n>m$.\\
 In the aforementioned papers, Jacquet and Shalika work with generic representations, however, their proofs extend 
verbatim to unramified representations of Langlands' type. Indeed, the formula for $W_{\pi}^0$ in terms of Satake parameters 
is still valid, and thanks to Proposition \ref{pairlanglands}, the factor $L(\pi,\pi',s)$ is still equal to 
(using notations of \cite{JS} and \cite{JS2}) the Artin 
factor $1/det(1-q^{-s}A\otimes A')$ for $A$ and $A'$ diagonal matrices corresponding to the Satake parameters 
of $\pi$ and $\pi'$ respectively. In particular, when $\pi$ is generic, $W_{\pi}^0$ is the essential vector of $\pi$.\\

Now let $\pi=\D_1\times \dots \times \D_t$ be a generic representation of $G_n$, written as a unique product of the 
unlinked segments $\D_i=[\nu^{-(k_i(\pi)-1)}\rho_i(\pi), \dots,\rho_i(\pi)]$, and 
$\pi'=\mu_1\times \dots \times \mu_{m}$ be an unramified 
representation of $G_{m}$ of Langlands' type, for $1\leq m\leq n$. One has, according to Proposition \ref{pairlanglands} and 
Theorem 8.2. of \cite{JPS2} (whose proof is independant of \cite{JPS}), the equality of Rankin-Selberg $L$-functions 
$L(\pi,\pi',s)=\prod_{i,j} L(\rho_i(\pi),\mu_j,s)$.\\
We notice that $L(\rho_i(\pi),\mu_j,s)$ is equal to $1$ unless $\rho_i(\pi)$ is an unramified character of $G_1$.
Hence, one has the equality 
$$L(\pi,\pi',s)=\prod_{\{i,\rho_i(\pi)\in  \widehat{F^*/\o^*},j\}}L(\rho_i(\pi),\mu_j,s).$$ 
This incites us to introduce the following representation.

\begin{df}\label{piu} Let $\pi= \D_1\times \dots \times \D_t$ be a generic representation of $G_n$, with 
$\D_i=[\nu^{-(k_i(\pi)-1)}\rho_i(\pi), \dots,\rho_i(\pi)]$. 
Let $r$ be the cardinality of the set $\{\rho_j(\pi), \rho_j(\pi)\in  \widehat{F^*/\o^*}\}$. When this set is non empty, denote by 
$\chi_1,\dots,\chi_r$ its elements ordered such that $Re(\chi_i)\geq Re(\chi_{i+1})$ for $1\leq i \leq r-1$. 
We define $\pi_u$ as the trivial representation of $G_0$ when $r=0$, and as the unramified representation of Langlands type
$\chi_1\times \dots \times \chi_r$ of $G_r$ when $r\geq 1$.\end{df}
 
Let $\pi$ a generic representation of $G_n$, and $\pi'$ be an 
unramified representation of $G_m$ of Langlands' type, with $1\leq m \leq n$.  If we set $L(\pi_u,\pi',s)=1$ when $\pi_u$ 
is the trivial representation of $G_0$, we have, according to Proposition \ref{pairlanglands}:
 \begin{equation}\label{test3}L(\pi,\pi',s)= L(\pi_u,\pi',s).\end{equation}
 From now on, we will order the segments $\D_i$ in 
the generic representation $\pi$, such that $\rho_i(\pi)$ is an unramified character $\chi_i$ of $G_1$ for $1\leq i \leq r$, is 
not such a character for $i\geq r+1$, and $Re(\chi_i)\geq Re(\chi_{i+1})$ for $1\leq i \leq r-1$.

\section{Mirabolic restriction, sphericity, and restriction of Whittaker functions}\label{restriction}

In this section, we first give results on the derivative functors and how they act on subspaces fixed by compact subgroups, then we recall 
some results from \cite{CP} about their interpretation in terms of restriction of Whittaker functions.
We introduce a few more notations, in order to get a handy parametrisation of the diagonal torus of $G_n$, in terms of simple roots. For $k\leq n$, 
let $Z_k$ be the center of $G_k$ naturally embedded in $G_n$; we parametrise it by 
$F^*$ using the morphism $\beta_k: z_k \mapsto diag(z_kI_k,I_{n-k})$. Hence the maximal torus $A_n$ of 
$G_n$ is the direct product $Z_1.Z_2\dots Z_{n-1}.Z_n$. We will sometimes (but not always) omit the $\beta_k$'s in 
this parametrisation and write $(z_1,\dots,z_{n})$ for 
the element $\beta_1(z_1)\dots \beta_n(z_n)$ of $A_n$. Notice that the $i$-th simple root $\alpha_i$ has 
the property that $\alpha_i(z_1,\dots,z_n)=z_i$.

\subsection{Mirabolic restriction and sphericity}

We first give a corollary of Proposition \ref{baseder}, about a concrete interpretation of $\Phi^-$, when restricted to 
$\Phi^+(Alg(P_n))$. Property d) of the aforementioned proposition says that $\Phi^{-}$ sends $\Phi^{+}\tau$ surjectively onto a 
$P_{n}$-module isomorphic to $\tau$.
Writing $\Phi^{+}\tau$ as $Ind_{P_{n}U_{n+1}}^{P_{n+1}}(\delta_{P_{n+1}}^{1/2}\tau \otimes \theta)$, we want to 
make the map $\Phi^{-}$ explicit between $Ind_{P_{n}U_{n+1}}^{P_{n+1}}(\delta_{P_{n+1}}^{1/2}\tau \otimes \theta)$ and $\tau$.

\begin{prop}\label{phi-}
For $n \geq 1$, if $\tau$ belongs to $Alg(P_n)$, then $\Phi^{-}$ identifies with the map 
$f\mapsto f(I_{n+1})$ from $\Phi^{+}\tau$ to $\tau$. 
\end{prop} 
\begin{proof}
Call $E_I$ the map $f\mapsto f(I_{n+1})$ from $\Phi^{+}\tau$ to $\tau$. Let's first show that $E_I$ is surjective. If 
$v$ belongs to the space $V$ of $\tau$, let $U =I_n+\mathcal{M}(n,\p^l)$ be a congruence subgroup of $G_n$, with $l$ large enough 
for $v$ to be fixed by $U'=U\cap P_n$. Call 
$f$ the function from $P_n$ to $V$, defined by $$f\begin{pmatrix} pu & x \\ & 1\end{pmatrix}= \delta_{P_{n+1}}^{1/2}(p)\theta(x_n)\tau(p)v,$$
 for $p \in P_n$, $u \in U$, $x\in 
F^n$ with bottom coordinate $x_n$, and $f(p')=0$ when $p'$ is not in $P_n U_{n+1} U$. The map $f$ is well defined, 
because $v$ is fixed by $U'$. It is smooth as it is right invariant under $U_{n+1}(\o) U$ (as $\theta$ is trivial on $\o$). One checks (see the proof of next Propositon for 
a detailed similar computation), that $f$ satsifies the requested left invariance under $P_nU_{n+1}$, so that $f\in \Phi^+(V)$. 
Finally $f(I_{n+1})=v$, thus $f$ 
is a preimage of $v$ via $E_I$.\\
 An easy adaptation of 
the Proposition 1.1. of \cite{CP} then shows that the $P_n$-submodule $\Phi^{+}\tau (U_{n+1},\theta)$ of $\Phi^+(\tau)$ is equal to 
$Ker(E_I)$. As a consequence, the map $E_I$ induces an isomorphism $\overline{E_I}$ between 
$\Phi^- \Phi^+ (\tau)$ and $\tau$, which is $P_n$-equivariant. Hence, the following diagram, with the right isomorphism equal to $\overline{E_I}$, commutes:
$$\begin{array}{lcc}
 \Phi^+(\tau) & \overset{\Phi^-}{\twoheadrightarrow}  &  \Phi^- \Phi^+(\tau) \\
  \ \ \ \upequal   &     & \downsimeq    \\
 \Phi^+(\tau)&  \overset{E_I}{\twoheadrightarrow}  &  \tau 
\end{array}$$
\end{proof}

We recall that for $n\geq 1$, as a consequence of the Iwasawa decomposition, any element $g$ of $G_n$ can be written in the form $zpk$ with 
$z$ in $F^*$, $p$ in $P_n$, and $k$ in $G_n(\o)$. 
We now notice that the restriction of $\Phi^{-}$ to $(\Phi^+\tau)^{P_{n+1}(\o)}$ is surjective onto $\tau^{P_{n}(\o)}$.

\begin{prop}\label{phi-spher}
For $n\geq 1$, the map $f\mapsto f(I_{n+1})$ from $(\Phi^{+}\tau)^{P_{n+1}(\o)}$ to $\tau^{P_{n}(\o)}$ is surjective.
\end{prop} 
\begin{proof}  
Let $v_0$ be a vector in the space of $\tau$ which is ${P_{n}(\o)}$-invariant, we claim that the function 
$f$ defined by 
$$f\begin{pmatrix} zpk & x \\ & 1\end{pmatrix}= \delta_{P_{n+1}}^{1/2}(p)\theta(x_n)\mathbf{1}_{\o^*}(z)
\tau(p)v_0,$$ for $z$ in $F^*$, $p$ in $P_{n}$, $k$ in ${G_{n}(\o)}$, 
and $x$ in $F^{n}$ (with $x$ the transpose of $(x_1,\dots,x_n)$), is a preimage of $v_0$ in $(\Phi^{+}\tau)^{P_{n+1}(\o)}$.\\
First we check that $f$ is well-defined: if $zpk=z'p' k'$, this implies that $z'$ is equal to $z$ mod $\o^*$, and 
$p'$ is equal to $p$ mod $P_n(\o)$. Hence 
$\delta_{P_{n+1}}^{1/2}(p)=\delta_{P_{n+1}}^{1/2}(p')$, $\mathbf{1}_{\o^*}(z')=\mathbf{1}_{\o^*}(z)$, 
and $\tau(p')v_0=\tau(p )v_0$ as $v_0$ is 
${P_{n}(\o)}$-invariant.\\
Then we check that $f$ indeed belongs to $\Phi^+(\tau)$. Let $p_0$ belong to $P_n$ embedded in $P_{n+1}$ by $p\mapsto diag(p,1)$, let 
$u_0$ belong to $U_{n+1}$, and $p_1$ belong to $P_{n+1}$, 
we need to check the relations 
$$f(p_0 p_1)=\delta_{P_{n+1}}^{1/2}(p_0 )\tau(p_0)f(p_1)$$ and 
$$f(u_0 p_1)= \theta(u_0)f(p_1).$$ 
Write $p_1= \begin{pmatrix}  zpk & x \\ & 1\end{pmatrix}$, we have $f(p_0 p_1)= f \begin{pmatrix}  
zp_0pk & p_0x \\ & 1\end{pmatrix}$. As 
$\theta((p_0x)_n)=\theta(x_n)$, we have 
$$f(p_0 p_1)= \delta_{P_{n+1}}^{1/2}(p_0p)\theta(x_n)\mathbf{1}_{\o^*}(z)
\tau(p_0p )v_0=\delta_{P_{n+1}}^{1/2}(p_0)\tau(p_0)f(p_1)$$ as wanted.\\
Write $u_0= \begin{pmatrix} I_n & x_0 \\ & 1\end{pmatrix}$, we also have 
$$f(u_0 p_1)= f \begin{pmatrix} zpk & x_0+ x \\ & 1\end{pmatrix}$$ 
$$= \delta_{P_{n+1}}^{1/2}(p)\theta((x_0+x)_n)\mathbf{1}_{\o^*}(z)
\tau(p )v_0=\theta(u_0)f(p_1)$$ as wanted.\\  
If $\begin{pmatrix} k_0  & x_0 \\ & 1\end{pmatrix}$ belongs to $P_{n+1}(\o)$ (i.e. $k_0\in G_n(\o)$ and $x_0\in \o^n$), as 
$\theta((zpkx_0)_n)=1$ when $z$ belongs to $\o^*$, 
we have $$f (\begin{pmatrix} zpk &   x \\ & 1\end{pmatrix}\begin{pmatrix} k_0  & x_0 \\ & 1\end{pmatrix})=
f \begin{pmatrix} zpkk_0 &  zpkx_0+ x \\ & 1\end{pmatrix}= f \begin{pmatrix}  zpk &   x \\ & 1\end{pmatrix},$$ 
and $f$ is right 
$P_{n+1}(\o)$-invariant. Finally, it is obvious that $f(I_{n+1})=v_0$.
\end{proof}

Now we are able to prove the following property of $\Phi^{-}$, that we will be of great use later.

\begin{prop}\label{spher1}
For $n\geq 1$, if $\tau$ belongs to $Alg(P_n)$, then $\Phi^{-}$ maps $\tau^{P_n(\o)}$ surjectively onto $\tau_{(1)}^{P_{n-1}(\o)}$, and 
$\Psi^{-}$ maps $\tau^{G_{n-1}(\o)}$ surjectively onto ${\tau^{(1)}}^{G_{n-1}(\o)}$.
\end{prop}
\begin{proof}
For the first part, we use the filtrations $0\subset \tau_{n} \subset \dots \subset \tau_{1}=\tau$ of $\tau$, and 
$0\subset \tau_{(1),n-1} \subset \dots \subset \tau_{(1),1}=\tau_{(1)}$ of $\tau_{(1)}$. But $\tau_{i}$ equals 
$\Phi^+(\tau_{(1),i-1})$ because of Lemma \ref{filtr}, so that $\Phi^-$ maps $\tau_{i}^{P_n(\o)}$ onto $\tau_{(1),i-1}^{P_{n-1}(\o)}$ 
surjectively according to Proposition \ref{phi-spher}. In particular, $\Phi^-$ maps $\tau_{2}^{P_n(\o)}$, hence $\tau^{P_n(\o)}$ (as 
$\tau_2^{P_n(\o)}\subset \tau^{P_n(\o)}$), onto $\tau_{(1),1}^{P_{n-1}(\o)}=\tau_{(1)}^{P_{n-1}(\o)}$ surjectively.\\
$\Psi^{-}$ maps $\tau^{G_{n-1}(\o)}$ surjectively onto ${\tau^{(1)}}^{G_{n-1}(\o)}$, because $\Psi^-$ is surjective from 
$\tau$ to $\tau^{(1)}$, and the functor $V\mapsto V^{G_{n-1}(\o)}$ is exact from $Alg(G_{n-1})$ to $Alg(G_0)$ as $G_{n-1}(\o)$ is compact open in $G_{n-1}$ ($\tau$ is a $G_{n-1}$-module by restriction).
\end{proof}

\subsection{Mirabolic restriction for Whittaker functions}

We start by recalling Proposition 1.1 of \cite{CP}, which gives an interpretation of $\Phi^-$ in terms of restriction of Whittaker functions.

\begin{prop}\label{stab}
For $k\geq 2$, and any submodule $\tau$ of $(\rho, C^{\infty}(N_k\backslash P_k,\theta))$ (where $\rho$ denotes the action of $P_k$ by right translation), 
the map $R:W\mapsto \delta_{P_k}^{-1/2}W_{|P_{k-1}}$ is  $P_{k-1}$-equivariant from 
$(\rho, C^{\infty}(N_k\backslash P_k,\theta))$ to $(\rho, C^{\infty}(N_{k-1}\backslash P_{k-1},\theta))$, with kernel 
$\tau (U_k,\theta)$. Hence it inducues a $P_{k-1}$-modules isomorphism between $\Phi^{-}\tau$ and 
$Im(R)\subset C^{\infty}(N_{k-1}\backslash P_{k-1},\theta)$, so that $(\rho, Im(R))$ is a model for $\Phi^{-}\tau$.
\end{prop}

Notice that for $k\geq 2$, if $g\in G_{k-1}$ equals $zpk$ with $z\in F^*,p\in P_{k-1}$, and $k\in G_{n-1}(\o)$, 
then the absolute value of $z$ depends only on $g$, so we can write it $|z(g)|$. We now state a proposition that follows 
from the proofs of Proposition 1.6. of \cite{CP}, about the interpretation of $\Psi^-$ in terms of Whittaker functions. 
\begin{prop}\label{psi-gen}
Let $\tau$ be a $P_k$-submodule of $C^{\infty}(N_k\backslash P_k,\theta)$, and suppose that $\tau^{(1)}$ is a $G_{k-1}$-module with 
central character $c$. Then, for any $W$ in 
$\tau$, for any $g$ in $G_{k-1}$, the quantity $c^{-1}(z)|z|^{-(k-1)/2}W(diag(zg,1))$ is constant whenever $z$ is in a punctured neighbourhood 
of zero (maybe depending on $g$) in $F^*$.\end{prop}

\begin{rem}Notice that in the proof of the Proposition 1.6. of \cite{CP}, $\tau$ is of a particular form, and $\tau^{(1)}$ is supposed 
to be irreducible. The only fact that is actually needed is that $\tau^{(1)}$ has a central character.\end{rem}

This has the following consequence.
    
\begin{cor}\label{psi-gen2}
For $k\geq 2$, let $\tau$ be a $P_k$-submodule of $(\rho,\mathcal{C}^{\infty}(N_k\backslash P_k,\theta))$, and suppose that $\tau^{(1)}$ is a 
$G_{k-1}$-module with central character $c$, then  $\Psi^-$ identifies with the map 
$$S:W \mapsto [g\mapsto \underset{z\rightarrow 0}{lim}\ c^{-1}(z)|z|^{(1-k)/2}W(diag(zg,1))\delta_{P_k}^{-1/2}(g)]$$ from $\tau$ to 
to $\mathcal{C}^{\infty}(N_{k-1}\backslash G_{k-1},\theta)$. To be more precise, $S$ has kernel $\tau (U_k,1)$, and it induces a $G_{k-1}$-modules 
isomorhism $\overline{S}$ between $\tau^{(1)}$ and $S(\tau)\subset \mathcal{C}^{\infty}(N_{k-1}\backslash G_{k-1},\theta)$.\end{cor}
\begin{proof}
For $W$ in $\tau$, call $\overline{W}$ its image in $\tau^{(1)}$, and call $S(W)$ the function 
$$[g\mapsto \underset{z\rightarrow 0}{lim}\ c^{-1}(z)|z|^{(1-k)/2}W(diag(zg,1))\delta_{P_k}^{-1/2}(g)]$$ in $C^{\infty}(N_{k-1}\backslash G_{k-1},\theta)$, 
which is well defined according to Proposition \ref{psi-gen}. If $u(x)=\begin{pmatrix} I_{k-1} & x \\ & 1 \end{pmatrix}$ belongs to 
$U_{k}$, then $\rho(u(x))W(diag(zg,1))= \theta (z (gx)_{k-1})W(diag(zg,1))$. As $\theta (z (gx)_{k-1})=1$ for $z$ small enough, 
we deduce that $S(\rho(u(x))W)=S(W)$, hence the kernel of $S$ contains $\tau (U_k,1)$. Conversely, if $S(W)=0$, the smoothness of 
$W$ and the Iwasawa decomposition imply that $W(g)$ is null for $|z(g)|$ in a punctured neighbourhood of zero depending only on $W$. According to 
Proposition 2.3. of \cite{M} (which is a restatement of Proposition 1.3. of \cite{CP}), 
this means that $W$ belongs to $\tau (U_k,1)$.\\ 
The $\C$-linear map $S:W \mapsto S(W)$ induces a $\C$-linear isomorphism 
$\overline{S}:\overline{W} \mapsto S(W)$ between $\tau^{(1)}$ and its image in $(\rho,\mathcal{C}^{\infty}(N_{k-1}\backslash G_{k-1},\theta))$. 
Moreover, it is a $G_{k-1}$-equivariant because for $g_0\in G_{k-1}$, one has that $\overline{S}(\tau^{(1)}(g_0)\overline{W})$ equals  
$$\overline{S}(\delta_{P_k}^{-1/2}(g_0)\overline{\rho(g_0)W})=\delta_{P_k}^{-1/2}(g_0)S(\rho(g_0)W)=\rho(g_0)S(W)=\rho(g_0)\overline{S}(\overline{W}).$$ 
\end{proof}

We end this section by stating two technical lemmas about Whittaker functions fixed under a maximal compact subgroup, the first 
is inspired from Lemma 9.2 of \cite{JPS2}.  

\begin{LM}\label{goodrestriction}
For $n\geq 2$, let $\tau$ be a $P_n$-submodule of $\mathcal{C}^{\infty}(N_n\backslash P_n,\theta)$, and let
$W$ belong to $\tau^{P_n(\o)}$, then there exists $W'$ in $\tau^{P_n(\o)}$, such that 
$W'(p\beta_{n-1}(z))= W(p)\mathbf{1}_{\o^*}(z)$ for $p$ in $P_{n-1}$ and $z$ in $F^*$.\end{LM}
\begin{proof}
For $l$ in $\Z$, we denote by $\phi_l$ the characteristic function $\mathbf{1}_{\p^l}$ of $\p^l$. We fix a Haar measure $dt$ on $F$. The 
Fourier transform $\widehat{\phi_l}^\theta$ with respect to $\theta$ and $dt$ is equal to $\l_l\phi_{-l}$ for $\l_l=dt(\p^l)>0$. 
We denote by $\Phi_l$ the function $\otimes_{i=1}^{n-1}\phi_l$, which is the characteristic of the lattice $\w^l\o^{n-1}$ in 
$F^{n-1}$. We denote by $u$ the natural isomorphism between $F^{n-1}$ and $U_n$. We also recall 
that any element of $\tau$ is determined by its restriction to $G_{n-1}$.\\
We set $W^l(p)= \int_{x\in F^{n-1}} W(pu(x))\Phi_l(x)dx$ for $p$ in $P_n$ and $dx=dt_1\otimes \dots \otimes dt_{n-1}$, hence $W^l$ belongs to 
$\tau$. Moreover if $k$ belongs to $G_{n-1}(\o)$, and $g$ belongs to $G_{n-1}$, then $W^l(gk)$ 
is equal to $\int_{x\in F^{n-1}} W(gku(x))\Phi_l(x)dx= \int_{x\in F^{n-1}} W(gu(kx))\Phi_l(x)dx$ because $W$ is 
${P_n(\o)}$-invariant, and this last integral is equal to 
$\int_{x\in F^{n-1}} W(gu(x))\Phi_l(k^{-1}x)dx= \int_{x\in F^{n-1}} W(gu(x))\Phi_l(x)dx=W^l(g)$ because of the invariance of $dx$ and 
$\Phi_l$ under $G_{n-1}(\o)$. It is also clear that $W^l$ is invariant $U_n(\o)$ because $W$ is, hence 
$W^l$ belongs to $\tau^{P_n(\o)}$.\\
Now, for $p$ in $P_{n-1}\subset P_n$, and $z$ in $F^*$, we obtain:
$$\begin{aligned} W^l(p\beta_{n-1}(z))& =  \int_{x\in F^{n-1}}  W(p\beta_{n-1}(z)u(x))\Phi_l(x)dx  \\
  & =\int_{x\in F^{n-1}}  W(u(zpx)p\beta_{n-1}(z))\Phi_l(x)dx  \\ 
  & =\int_{x\in F^{n-1}}  \theta((zpx)_{n-1})W(p\beta_{n-1}(z))\Phi_l(x)dx\\
  & =W(p\beta_{n-1}(z))\int_{x\in F^{n-1}} \Phi_l(x)\theta(zx_{n-1})dx  \\
  & = W(p\beta_{n-1}(z))\int_{y\in \w^l\o^{n-2}} dy\int_{t\in F} \phi_l(t)\theta(zt)dt\\
  &= \l_l^{n-2}W(p\beta_{n-1}(z))\widehat{\phi_l}^\theta(z) = \l_l^{n-1}W(p\beta_{n-1}(z))\phi_{-l}(z)\end{aligned}$$
 
\noindent The function $W'=W^0/\l_0^{n-1}-W^{-1}/\l_{-1}^{n-1}$ thus satisfies 
$$W'(p\beta_{n-1}(z))\!=\! W(p\beta_{n-1}(z))(\phi_0-\phi_{1})(z)
\!= \! W(p\beta_{n-1}(z))\mathbf{1}_{\o^*}(z)\!=\! W(p)\mathbf{1}_{\o^*}(z)$$
because $W$ is invariant under $\beta_{n-1}(\o^*)\subset P_n(\o)$. 
\end{proof}

\begin{LM}\label{samegerm}
For $n \geq 2$, let $\tau$ be a $P_n$-submodule of $\mathcal{C}^{\infty}(N_n\backslash P_n,\theta)$, and let
$W$ belong to $\tau^{G_{n-1}(\o)}$, then there exists $W'$ in $\tau^{P_n(\o)}$, such that 
$W'(z_1,\dots,z_{n-1},1)= W(z_1,\dots,z_{n-1},1)1_{\o}(z_{n-1})$ for $z_i$ in $F^*$.\end{LM}
\begin{proof} Let $du$ be the Haar measure on $U_n$, corresponding to the Haar measure $dx=dt_1\otimes\dots \otimes dt_{n-1}$ on $F^{n-1}$, 
normalised by $dt_i(\o)=1$ for $i$ between $1$ and $n-1$. Now set $W'(g)=\int_{u\in U_n(\o)} W(gu)du$. The vector $W'$ is a linear combination of 
right translates of $W$ by elements of $U_n(\o)$, so it belongs to $\tau$.
 It is clearly invariant under $U_n(\o)$, and still invariant under 
$G_{n-1}(\o)$, as $G_{n-1}(\o)$ normalises $U_n(\o)$. The following computation then gives the result:
$$W'(z_1,\dots,z_{n-1},1)=\int_{x\in {\o}^{n-1}} W(\beta_1(z_1)\dots \beta_{n-1}(z_{n-1})u(x))dx$$ 
$$=\int_{x\in {\o}^{n-1}}
 W(u(\beta_1(z_1)\dots \beta_{n-1}(z_{n-1})x)\beta_1(z_1)\dots \beta_{n-1}(z_{n-1}))dx$$
$$=\int_{x\in {\o}^{n-1}} \theta (z_{n-1}x_{n-1}) W(\beta_1(z_1)\dots \beta_{n-1}(z_{n-1}))dx$$
$$= \widehat{\1_{\o}}^\theta(z_{n-1}) W(z_1,\dots,z_{n-1},1)=\1_{\o}(z_{n-1})   W(z_1,\dots,z_{n-1},1),$$ 
the last equality because of the normalisation of the Haar measure on $F$.
\end{proof}

\section{Construction of the essential Whittaker function}

We are now going to produce the essential vector of a generic representation 
$\pi$ of $G_n$, which will now be fixed untill the end. We recall recall that we associated to $\pi$, an integer 
$0\leq r\leq n$, and an unramified representation of Langlands' type $\pi_u$ 
of $G_r$ in Section \ref{genunram}.\\

We first notice that the subspace of $\pi^{(n-r)}$ fixed by $G_r(\o)$ is a complex line.

\begin{prop}\label{sousmoduleder}
Let $\pi$ be generic representation of $G_n$. Then $(\pi^{(n-r)})^{G_r(\o)}$ is of dimension $1$. If $v^0$ is a generator of 
$(\pi^{(n-r)})^{G_r(\o)}$, then the submodule $<G_r.v^0>$ of $\pi^{(n-r)}$ spanned by $v^0$ surjects onto $\pi_u$.
\end{prop}
\begin{proof}
Write $\pi=\D_1\times \dots \times \D_t$ for the ordering of the $\D_i$'s fixed after Definition \ref{piu}. 
According to Lemma 3.5. of \cite{BZ} the representation $\pi^{(n-r)}$ 
has a filtration with subquotients $\D_1^{(a_1)}\times \dots \times \D_t^{(a_t)}$, with $\sum_i a_i = n-r$. According 
to Proposition 9.6 of \cite{Z}, $\pi_u$ appears as one of these subquotients, and by the choice of $r$, the other nonzero 
subquotients amongst them all contain either a segment as a factor (in the product notation) of some $G_k$ for $k\geq 2$, 
or a ramified character of $G_1$. According to Proposition \ref{segmentramified}, and Proposition \ref{parabspherical} \ref{parabspherical}, 
these other subquotients contain no nonzero $G_n(\o)$-invariant vector. The result then follows from the exactness of the functor 
$V\mapsto V^{G_n(\o)}$ from $Alg(G_n)$ to $Alg(G_0)$.
\end{proof}

 We also notice the following facts. First, from the theory of Kirillov models 
(see \cite{BZ}, Theorem 4.9), for $n\geq 2$, the map $W\in W(\pi,\theta)\mapsto W_{|P_n}$ is injective,
 we denote by $W(\pi_{(0)},\theta)$ its image. We choose this notation because  
$P_n$-module $\pi_{(0)}=\pi_{|P_n}$ is isomorphic to the submodule $W(\pi_{(0)},\theta)$ of 
$(\rho,\mathcal{C}^{\infty}(N_n\backslash P_n,\theta))$. Now if one applies Proposition \ref{stab} repeatedly to 
$\pi_{(0)}$, then for $r\leq n-1$, the $P_{r+1}$-module $\pi_{(n-r-1)}$ is isomorphic to the submodule of 
$(\rho,\mathcal{C}^{\infty}(N_{r+1}\backslash P_{r+1},\theta))$, whose vectors are the functions 
$(\prod_{k=r+2}^n\delta_{P_k}^{-1/2})W_{|P_{r+1}}$ for $W\in W(\pi,\theta)$ (where $P_{r+1}$ is embedded in $P_n$ via $p\mapsto 
diag(p,I_{n-r-1})$), we denote by 
$W(\pi_{(n-r-1)},\theta)$ this $P_{r+1}$-module.\\ 

Proposition \ref{sousmoduleder} has the following corollary.

\begin{cor}\label{sousmodulegen}
Under the condition $1\leq r \leq n-1$, there exists $\tilde{W_0}$ in $W(\pi_{(n-r-1)},\theta)^{G_r(\o)}$ such that 
$$W_{\pi_u}^0(g)= \underset{z\rightarrow 0}{lim}\ c^{-1}(z)|z|^{-r/2}\tilde{W_0}(diag(zg,1))\delta_{P_{r+1}}^{-1/2}(g)$$ for all $g$ in $G_r$. 
This implies that the representation $\pi_u$ occurs as a submodule of $\pi^{(n-r)}$.\end{cor}
\begin{proof} 
We take $\Psi^-(W(\pi_{(n-r-1)},\theta))$ as a model for $\pi^{(n-r)}$, i.e. 
$\pi^{(n-r)}=\Psi^-(W(\pi_{(n-r-1)},\theta))$. Let $\tilde{W}_0$ be a preimage (which we shall normalise later), of 
$v^0$ under $\Psi^-$, which we take in $W(\pi_{(n-r-1)},\theta)^{G_r(\o)}$ thanks to Proposition \ref{spher1}. We denote by 
$<P_{r+1}.\tilde{W}_0>$ the $P_{r+1}$-submodule of $W(\pi_{(n-r-1)},\theta)$ spanned by $\tilde{W}_0$. By definition of $\Psi^-$, we have 
$$\Psi^-(<P_{r+1}.\tilde{W}_0>)=\Psi^-(<G_r.\tilde{W}_0>)=<G_r.v^0>.$$ Now, $Z_n$ acts by a character $c$ on $<G_r.v^0>$ 
(which is the central character of $\pi_u$ 
as well according to Proposition \ref{sousmoduleder}). 
Let $S$ be the map defined in Corollary \ref{psi-gen2} from 
$<P_{r+1}.\tilde{W}_0>$ to $\mathcal{C}^\infty(N_r\backslash G_r,\theta)$. We know from this corollary, 
that $S$ factors to give an isomorphism $\overline{S}$ between $<G_r.v^0>$ and $S(<P_{r+1}.\tilde{W}_0>)$. 
Define $W^0$ as $W^0=S(\tilde{W}_0)$ in $\mathcal{C}^\infty(N_r\backslash G_r,\theta)$, so that $S(<P_{r+1}.\tilde{W}_0>)$ is equal to $G_r.W^0$. 
As the $G_r$-module $<G_r.W^0>$ is isomorphic to $<G_r.v^0>$, there is a surjective 
$G_r$-module morphism from $<G_r.W^0>$ onto $W(\pi_u,\theta)$ according to Proposition \ref{sousmoduleder}. It sends $W^0$ to a nonzero multiple of 
$W_{\pi_u}^0$. We normalise $\tilde{W}_0$, such that the Whittaker 
function $W^0=S(\tilde{W}_0)$ is equal to $1$ on $G_n(\o)$. The Hecke algebra $\mathcal{H}_r$ thus multiplies 
$W^0$ and $W_{\pi_u}^0$ by the same character, as $W_{\pi_u}^0$ is the image of $W^0$ via a $G_r$-intertwining operator. Both are 
normalised spherical Whittaker functions, they are thus equal according to \cite{S}. In particular, we have 
$S(\tilde{W}_0)=W_{\pi_u}^0$, which is the first statement of the corollary. Next, this implies the equalities 
$<G_r.W^0>=<G_r.W_{\pi_u}^0>=W(\pi_u,\theta)$, so the surjection from $<G_r.v^0>$ onto $W(\pi_u,\theta)$ is actually 
equal to the isomorphism $\overline{S}$, hence $\pi_u$ occurs as a submodule of $\pi^{(n-r)}$. 
\end{proof}

The following proposition then holds.

\begin{prop}
Under the condition $1\leq r \leq n-1$, there exists in $W(\pi_{(n-r-1)},\theta)^{P_{r+1}(\o)}$ an element $W_0$, such that 
$W_0(z_1,\dots,z_r,1)=\delta_{P_{r+1}}^{1/2}(z_1,\dots,z_r)W_{\pi_u}^0(z_1,\dots, z_r)1_{\o}(z_r)$ for $z_r$ in $F^*$.
\end{prop}
\begin{proof}
Let $\tilde{W}_0$ be as in Corollary \ref{sousmodulegen}. By the claim in the proof of Theorem 2.1 of \cite{M} (the arguments of 
the claim in Proposition 1.6. of \cite{CP} are actually sufficient here), 
there is $N$ in $\N$, such that $\tilde{W}_0(z_1,\dots,z_r a,1)=c_{\pi_u}(a)|a|^{r/2} \tilde{W}_0(z_1,\dots,z_r,1)$ 
(parametrizing $A_{r+1}$ with the $\beta_i$'s) for $|z_r|\leq q^{-N}$ and $|a|\leq 1$. For $b$ in $F^*$, call $\tilde{W}_{0,b}$ the function 
$p\mapsto \tilde{W}_0(p\beta_{r}(b))/(c_{\pi_u}(b)|b|^{r/2}) $ defined on $P_{r+1}$, then $\tilde{W}_{0,b}$ still belongs to 
$W(\pi_{(n-r-1)},\theta)^{G_r(\o)}$, and $\tilde{W}_{0,b}(z_1,\dots,z_r,1)/(c_{\pi_u}(z_r)|z_r|^{r/2})$ is constant 
with respect to $z_r$ whenever $|z_r|\leq q^{-N}/|b|$. We choose $b$ in $F^*$ satisfying $|b|=q^{-N}$, so that the function 
$\tilde{W}_{0,b}(z_1,\dots,z_r,1)/(c_{\pi_u}(z_r)|z_r|^{r/2})$ is constant with respect to $z_r$ for $|z_r|\leq 1$. Hence, 
according to Corollary \ref{sousmodulegen}, for $|z_r|\leq 1$, we have the equalities 
$$\tilde{W}_{0,b}(z_1,\dots,z_r,1)/(c_{\pi_u}(z_r)|z_r|^{r/2})=\tilde{W}_{0}(z_1,\dots,z_rb,1)/(c_{\pi_u}(z_rb)|z_rb|^{r/2})$$
$$=\delta_{P_{r+1}}^{1/2}(z_1,\dots,z_{r-1},b)W_{\pi_u}^0(z_1,\dots,z_{r-1},b)/(c_{\pi_u}(b)|b|^{r/2})$$
$$=\delta_{P_{r+1}}^{1/2}(z_1,\dots,z_{r-1},1)W_{\pi_u}^0(z_1,\dots,z_{r-1},1).$$ They imply the equality 
$$\tilde{W}_{0,b}(z_1,\dots,z_r,1)=\delta_{P_{r+1}}^{1/2}(z_1,\dots,z_{r-1},z_r)W_{\pi_u}^0(z_1,\dots,z_{r-1},z_r)$$ for $|z_r|\leq 1$. 
On the other hand, applying Lemma \ref{samegerm}, 
there is $W_0$ is in $W(\pi_{(n-r-1)},\theta)^{P_{r+1}(\o)}$, such that $W_0(z_1,\dots,z_r,1)$ is equal to 
$\tilde{W}_{0,b}(z_1,\dots,z_r,1)1_{\o}(z_r)$, it is then clear that $W_0$ has the desired property.  
\end{proof}

We now prove the main result of this paper.

\begin{thm}\label{essential}
For $n\geq 2$, let $\pi$ be a ramified generic representation of $G_n$ (i.e. $r\leq n-1$). Then one can produce a $G_{n-1}(\o)$-invariant function 
$W_{\pi}^{ess}$ in $W(\pi,\theta)$, whose restriction 
to $A_{n-1}$ (when $A_{n-1}$ is parametrised by its simple roots), is given by formula 

$$W_{\pi}^{ess}(z_1,\dots,z_{n-1},1)$$
\begin{equation}\label{formule2}= W_{\pi_u}^0(z_1,\dots,z_r)\nu(z_1,\dots,z_r)^{(n-r)/2}\mathbf{1}_{\o}(z_r)
\prod_{j=r+1}^{n-1}\1_{\o^*}(z_j)\end{equation} when $r\geq 1$, and by 
\begin{equation}\label{formule2'} W_{\pi}^{ess}(z_1,\dots,z_{n-1},1)= \prod_{j=1}^{n-1}\1_{\o^*}(z_j)\end{equation} when $r=0$.
 A function $W_{\pi}^{ess}$ with such properties is unique, and has image $W_{\pi_u}^0$ in $\pi^{(n-r)}$.

\end{thm}
\begin{proof} Suppose first that we have $r\geq 1$. We already constructed in the previous proposition a vector $W_0$ in 
$W(\pi_{(n-r-1)},\theta)^{P_{r+1}(\o)}$ such that 
$$W_0(z_1,\dots,z_r,1)=\delta_{P_{r+1}}^{1/2}(z_1,\dots,z_r)W_{\pi_u}^0(z_1,\dots, z_r)1_{\o}(z_r).$$ 
Then, applying Proposition \ref{spher1} and then Lemma \ref{goodrestriction}, we obtain $W_1$ in 
$W(\pi_{(n-r-2)},\theta)^{P_{r+2}(\o)}$, that satisfies 
$$\begin{array}{l}W_1(z_1,\dots,z_{r+1},1)=\delta_{P_{r+2}}^{1/2}(z_1,\dots,z_{r+1})W_0(z_1,\dots,z_r,1)\1_{\o^*}(z_{r+1})\\
= \delta_{P_{r+2}}^{1/2}(z_1,\dots,z_{r},1)\delta_{P_{r+1}}^{1/2}(z_1,\dots,z_{r})W_{\pi_u}^0(z_1,\dots, z_r)1_{\o}(z_r)\1_{\o^*}(z_{r+1}).\end{array}$$
Repeating this last step (Proposition \ref{spher1} and then Lemma \ref{goodrestriction}), we obtain by induction for all $k$ between 
$1$ and $n-r-1$, an element $W_k$ in $W(\pi_{(n-r-1-k)},\theta)^{P_{r+k+1}(\o)}$, that satisfies 
$$W_k(z_1,\dots,z_{r+k},1)=\delta_{P_{r+k+1}}^{1/2}(z_1,\dots,z_{r+k})W_{k-1}(z_1,\dots,z_{r+k-1},1)\1_{\o^*}(z_{r+k})$$ 
$$=W_{\pi_u}^0(z_1,\dots,z_r)\mathbf{1}_{\o}(z_r)\prod_{j=r+1}^{r+k}\1_{\o^*}(z_j)
\prod_{i=r+1}^{r+k+1}\delta_{P_i}^{1/2}(z_1,\dots,z_r,\underbrace{1,\dots,1}_{i-(r+1)\times}).$$
We define $W_{\pi}^{ess}$ to be the element of 
$W(\pi,\theta)$ which restricts to $P_n$ as $W_{n-r-1}$, it is thus $G_{n-1}(\o)$-invariant and satisfies Equation (\ref{formule2}) of the statement 
of the theorem, because 
$$\prod_{i=r+1}^{n}\delta_{P_i}^{1/2}(z_1,\dots,z_r,\underbrace{1,\dots,1}_{i-(r+1)\times})=|det(z_1,\dots,z_r)|^\frac{n-r}{2}$$
as $\delta_{P_i}(z_1,\dots,z_r, 1,\dots,1)=|det(z_1,\dots,z_r)|$ for $i>r$.\\
If $r=0$, we take for $W_0$ the constant function on the trivial group $P_1$ equal to $1$ in 
$W(\pi_{(n-1)},\theta)=W(\pi_{(n-1)},\theta)^{P_1(\o)}$. Again, thanks to 
Proposition \ref{spher1} and Lemma \ref{goodrestriction}, there is $W_1$ in $W(\pi_{(n-2)},\theta)^{P_2(\o)}$ 
such that $W_1(z_1,1)=\1_{\o^*}(z_1)$ for $z_1$ in $F^*$, and we end as in the case $r\geq 1$.\\  
 The function $W_{\pi}^{ess}$ is unique by the theory of 
Kirillov models, and its image in $\pi^{(n-r)}$ is $W_{\pi_u}^0$ by 
construction.
\end{proof}

The expression of the restriction of $W_\pi^{ess}$ to $A_{n-1}$ in the usual coordinates is the same.

\begin{cor}\label{corformule}
 Let $\pi$ be a ramified generic representation of $G_n$, then if $a=diag(a_1,\dots,a_{n-1})$ belongs to $A_{n-1}$, and 
$a'=diag(a_1,\dots,a_r)\in A_r$, we obtain Formulas (\ref{formule}) and (\ref{formule'}) of the introduction:
$$W_{\pi}^{ess}(diag(a,1))=W_{\pi_u}^0(a')\nu(a')^{(n-r)/2}\mathbf{1}_{\o}(a_r)
\prod_{j=r+1}^{n-1}\1_{\o^*}(a_j)$$ when $r\geq 1$, and $$W_{\pi}^{ess}(diag(a,1))=\prod_{j=1}^{n-1}\1_{\o^*}(a_j)$$ when $r=0$.

\end{cor}
\begin{proof}
We do the case $r\geq 1$, the case $r=0$ being simpler. 
If $diag(a_1,\dots,a_{n-1})=(z_1,\dots,z_{n-1})$ belongs to $A_{n-1}$, we have $a_i=z_i\dots z_{n-1}$, hence 
if $a'=diag(a_1,\dots,a_r)$, we have $(z_1,\dots,z_r)=(\prod_{i=r+1}^{n-1} z_i)^{-1}a'$ in $A_r$. 
Equation (\ref{formule2}) can thus be read: 
$$W_{\pi}^{ess} (diag(a,1))$$ 
$$=W_{\pi_u}^0((\prod_{i=r+1}^{n-1} z_i)^{-1}a')\nu((\prod_{i=r+1}^{n-1} z_i)^{-1}a')^{(n-1)/2}
\mathbf{1}_{\o}(z_r)
\prod_{i=r+1}^{n-1}\1_{\o^*}(z_i)$$
$$=W_{\pi_u}^0(a')\nu(a')^{(n-r)/2}\mathbf{1}_{\o}(a_r)
\prod_{i=r+1}^{n-1}\1_{\o^*}(a_i),$$ the last equality because if it is not $0=0$, this means that $z_{r+1},\dots,z_{n-1}$ 
all belong to $\o^*$, hence the inverse of their product as well, and $W_{\pi_u}^0$, $\nu$, $\1_{\o^*}$ and $\1_{\o}$ are 
all invariant under $\o^*$.
\end{proof}

We then have the following corollary.

\begin{cor}\label{testf}
Let $\pi$ be a generic representation of $G_n$ with Whittaker model $W(\pi,\theta)$. There exists 
in $W(\pi,\theta)$ a unique $G_{n-1}(\o)$-invariant function $W_{\pi}^{ess}$ equal to $1$ on $G_{n-1}(\o)$, such that for every $ 1 \leq m \leq n-1$, and every unramified
representation $\pi'$ of Langlands' type of $G_m$, with normalised spherical function $W_{\pi'}^0$ in $W(\pi',\theta^{-1})$, the equality 
$I(W_{\pi}^{ess},W_{\pi'}^0,s)= L(\pi,\pi',s)$ holds for an appropriate normalisation 
of the invariant measure on $N_{m}\backslash G_{m}$.\\
\end{cor}
\begin{proof} 
The unicity of a function $W_{\pi}^{ess}$ with such properties follows from \cite{JPS}. If $\pi$ is unramified (i.e. $r=n$), we set $W_{\pi}^{ess}=W_{\pi}^0$ and Equations (\ref{test1}) and (\ref{test2}) show that it is the correct choice.\\
 When $r\leq n-1$, we again only treat the case $r\geq 1$, the case $r=0$ being similar,
 but simpler (using Equation (\ref{formule'}) instead of Equation (\ref{formule})). We show that the function $W_\pi^{ess}$
  from the previous corollary satisfies the wanted equalities.\\ 
Thanks to Iwasawa decomposition, we have $$I(W_{\pi}^{ess}, W_{\pi'}^0,s)=\int_{A_{m}} W_{\pi}^{ess}(diag(a,I_{n-m}))W_{\pi'}^0(a) 
\delta_{B_{m}}^{-1}(a)\nu(a)^{s-\frac{(n-m)}{2}}d^*a'.$$  
If $m>r$, using Equations (\ref{formule}) and $\delta_{B_{m}}\begin{pmatrix}a & \\ &I_{m-r}\end{pmatrix}= \delta_{B_{m}}(a)\nu(a)^{m-r}$,
 we find
$$I(W_{\pi}^{ess}, W_{\pi'}^0,s)=\! \int_{A_{r}}\!\!\! W_{\pi_u}^{0}(a')
W_{\pi'}^0\begin{pmatrix} a' & \\ & I_{m-r} \end{pmatrix}\delta_{B_{r}}^{-1}(a')
\mathbf{1}_{\o}(a_r)\nu(a')^{s-\frac{(m-r)}{2}}d^*a'$$ 
$$=\int_{A_{r}} W_{\pi_u}^{0}(a')
W_{\pi'}^0\begin{pmatrix} a' & \\ & I_{m-r} \end{pmatrix}\delta_{B_{r}}^{-1}(a')\nu(a')^{s-\frac{(m-r)}{2}}d^*a'=I(W_{\pi'}^0,W_{\pi_u}^{0},s),$$ 
as $W_{\pi'}^0\begin{pmatrix} a' & \\ & I_{m-r} \end{pmatrix}$ vanishes for $|a_r|>1$. Hence, by Equations (\ref{test2}) and (\ref{test3}), we obtain  
$$I(W_{\pi}^{ess}, W_{\pi'}^0,s)=L(\pi_u,\pi',s)=L(\pi,\pi',s).$$
If $m=r$, using Equation (\ref{formule}), we find 
$$I(W_{\pi}^{ess}, W_{\pi'}^0,s) =\int_{A_{r}} W_{\pi_u}^{0}(a') W_{\pi'}^0\begin{pmatrix} a'\end{pmatrix}\delta_{B_{r}}^{-1}(a')
\mathbf{1}_{\o}(a_r)\nu(a')^{s}d^*a',$$ but this integral is equal to  
$$I(W_{\pi_u}^0, W_{\pi'}^0,\mathbf{1}_{\o^m}, s)=L(\pi_u,\pi',s)=L(\pi,\pi',s)$$ by Equations (\ref{test1}) and (\ref{test3}).\\
If $m<r$, Equation (\ref{formule}) gives $$I(W_{\pi}^{ess}, W_{\pi'}^0,s)=\int_{A_{m}} W_{\pi_u}^{0}(diag(a,I_{r-m}))W_{\pi'}^0(a) \delta_{B_{m}}^{-1}(a)\nu(a)^{s-\frac{(r-m)}{2}}d^*a,$$ 
and this integral is equal to  
$$I(W_{\pi_u}^0, W_{\pi'}^0)=L(\pi_u,\pi',s)= L(\pi,\pi',s)$$ by Equations (\ref{test2}) and (\ref{test3}).\\
In all cases, we have $$I(W_{\pi}^{ess}, W_{\pi'}^0,s)=L(\pi,\pi',s).$$ 
\end{proof}

\begin{ack} I thank the $p$-adic workgroup in Poitiers for allowing me to give talks about \cite{JPS}. I thank Guy Henniart, Herv\'e Jacquet, and Michitaka Miyauchi, for useful comments and corrections about this note. Finally and most importantly, I thank the referee for his very careful reading of the paper, which allowed to correct some mistakes, and improve many arguments.\end{ack}

\Addresses
\end{document}